\documentclass[12pt]{amsart}
\usepackage{amssymb}
\usepackage[all]{xy}
\usepackage{tabulary,caption}
\theoremstyle{plain}

\newtheorem{Thm}{Theorem}[section]

\newtheorem{Lem}[Thm]{Lemma}

\theoremstyle{definition}

\theoremstyle{remark}

\newtheorem{Exa}[Thm]{Example}

\numberwithin{equation}{subsection}

\begin{document}

\title[Factorization of Symplectic Matrices into elementary factors]%
{Factorization of Symplectic Matrices into elementary factors}
\author{Bj\"orn Ivarsson \and Frank Kutzschebauch \and Erik L{\o}w}
\address{Department of Mathematics of Systems Analysis\\
Aalto University\\
P.O. Box 11100, FI--00076 Aalto, Finland}
\address{Departement Mathematik\\
Universit\"at Bern\\
Sidlerstrasse 5, CH--3012 Bern, Switzerland}
\address{Department of Mathematics\\
University of Oslo\\
P.O. Box 1053, Blindern, NO--0316 Oslo, Norway}
\email{bjorn.ivarsson@aalto.fi}
\email{frank.kutzschebauch@math.unibe.ch}
\email{elow@math.uio.no}
\thanks{Part of this research was done while the authors were visitors at The Centre for Advanced 
Study (CAS) at the Norwegian Academy of Science and Letters. Bj\"orn Ivarsson was also 
supported by the Magnus Ehrnrooth Foundation and Erik L\o w by Bergens Forskningsstiftelse (BFS). The research of Frank Kutzschebauch  was partially supported by Schweizerische Nationalfonds Grant 200021-178730.}

\date{\today}
\setcounter{tocdepth}{3}
\begin{abstract} We prove that a symplectic matrix with entries in a ring with Bass stable
rank one can be factored as a product of elementary symplectic matrices. This also holds
for null-homotopic symplectic  matrices with entries in a Banach algebra or in the ring of complex valued
continuous functions on a finite dimensional normal topological space.
\end{abstract}
\maketitle
\bibliographystyle{amsalpha}
\tableofcontents
\section{Introduction and main results}\label{introduction}
In this paper $R$ denotes a commutative ring with identity, $\operatorname{SL}_{n}(R)$ the 
matrices of determinant $1$ with entries in $R$ and $E_{n}(R)$ the group generated by the elementary 
matrices. The problem of whether every matrix in $\operatorname{SL}_{n}(R)$ factors
as a product of elementary matrices, i.e. is an element of $E_{n}(R)$, has been 
studied extensively for various rings of polynomials and functions.
For a polynomial ring of one variable $R=k[x]$ the result is simple. For several variables
$R=k[x_1,\cdots,x_k]$ the result is not true for $n=2$ (\cite{Cohn:1966}) but by a famous 
result of Suslin (\cite{Suslin:1977}) it is true for $n \ge 3$. The second author and E.Doubtsov
recently proved that the result holds for rings with Bass stable rank $1$ (\cite{Doubtsov/Kutzschebauch:2019}) 
If $R$ is a unital commutative Banach algebra then every null-homotopic matrix in $\operatorname{SL}_n(R)$
is in $E_n(R)$ (\cite{Milnor:1971}). In the case of $R=C(X)$, the continuous complex functions
on a finite dimensional normal topological space, Vaserstein had previously proven the same
result for null-homotopic matrices (\cite{Vaserstein:1988}). Finally, the first two authors
(\cite{Ivarsson:2012}) 
proved the  result for null-homotopic matrices in the case of $R=\mathcal{O} (X)$, the holomorphic
functions on a reduced Stein space $X$, thus solving the so-called Vaserstein problem of 
Gromov (\cite{Gromov:1989}).  

The corresponding problem for the symplectic matrices, $\operatorname{Sp}_{2n}(R)$, has not been studied
to the same degree. The group generated by the elementary symplectic matrices is denoted 
by $\operatorname{Ep}_{2n}(R)$ (definitions will follow in Section 2). Again it follows easily that
$\operatorname{Sp}_{2n}(R) = \operatorname{Ep}_{2n}(R)$ for $R=k[x]$, this being a Euclidean ring. For $n \ge 2$ 
Kopeiko proved this for $R=k[x_1,\cdots,x_k]$ (\cite{Kopeiko:1978}) and Grunewald/Mennicke/
Vaserstein proved it for $R=\mathbb{Z}[x_1,\cdots,x_k]$. In this paper we take up the study for various
function spaces and we prove symplectic versions of the results in \cite{Milnor:1971}, 
\cite{Doubtsov/Kutzschebauch:2019} and \cite{Vaserstein:1988}. The Vaserstein problem
for null-homotopic holomorphic symplectic matrices turns out to be very complicated and
requires the use of Gromov's Oka principle for holomorphic sections of elliptic bundles
(\cite{Gromov:1989}). In
a forthcoming paper we solve the problem for $4 \times 4$ matrices. For one-dimensional
spaces $X$, however, the result is much easier and follows from our results here, for any
size matrix. More precisely, we will prove :

\begin{Thm}\label{t:Theorem 1}
If $R$ is a commutative Banach algebra with unity and  $M \in \operatorname{Sp}_{2n}(R)$ is null-homotopic,
then $M \in \operatorname{Ep}_{2n}(R)$.
\end{Thm}

\begin{Thm}\label{t:Theorem 2}
If $R$ has Bass stable rank $1$, then $\operatorname{Sp}_{2n}(R) = \operatorname{Ep}_{2n}(R)$.
\end{Thm}

\begin{Thm}\label{t:Theorem 3}
If $X$ is a finite dimensional normal topological space and $M \in \operatorname{Sp}_{2n}(C(X))$ is null-homotopic,
then $M \in \operatorname{Ep}_{2n}(C(X))$.
\end{Thm}

In Section 2 we will give definitions and some elementary observations. In Section 3 
we give examples and the remaining sections prove the theorems.

\section{Definitions}\label{definitions}

The symplectic group $\operatorname{Sp}_{2n}(R)$ is a subgroup of $\operatorname{SL}_{2n}(R)$. We shall write matrices
with block notation
\[ \begin{pmatrix} A & B \\ C & D \end{pmatrix} \]
where $ A, B, C $ and $ D $ are  $ (n \times n) $ matrices with entries in $R$ satisfying the symplectic
conditions
\begin{equation}\label{e:firstsymp}
A^TC=C^TA
\end{equation}
\begin{equation}\label{e:secondsymp}
B^TD=D^TB
\end{equation}
\begin{equation}\label{e:thirdsymp}
A^TD-C^TB=I
\end{equation}
where  $ I $ is the  $ (n\times n) $ identity matrix.

An \emph{elementary symplectic matrix} is either of the form
\[ \begin{pmatrix} I & B \\ 0 & I \end{pmatrix} \]
where $B$ is symmetric ($B=B^T$) or of the form 
\[ \begin{pmatrix} I & 0 \\ C & I \end{pmatrix} \]
where $C$ is symmetric. Products of matrices of the first type are
additive in $B$ and of the second type in $C$. Special cases are the
matrices $E_{ij}(a)$ when $B$ is the matrix with $a$ in position $ij$ and
$ji$ and otherwise zero. For $F_{ij}(a)$ the roles of $B$ and $C$ are
changed. Clearly any elementary matrix of the first type is the product
of matrices $E_{ij}(b_{ij})$ for $i \le j$ and similarly for  the second type.

We  notice that multiplying a matrix by $E_{ij}(a)$ from the left adds
$a$ times the (n+j)-th row to the i-th row and $a$ times the (n+i)-th row
to the j-th row. Multiplying by $F_{ij}(a)$ adds $a$ times the j-th row
to the (n+i)-th row and $a$ times the i-th row to the (n+j)-th row.

We also introduce the symplectic matrices $K_{ij}(a)$ defined by $B=C=0$ and
$A=I$ except in position ij, where there is an $a$. Finally, $D=(A^t)^{-1}$.
This equals $I$ except in position ji, where there is $-a$ if $i\ne j$ and 
$a^{-1}$ if $i=j$ (this requires $a \in R^*$). Multiplying a matrix $M$ from the left by $K_{ij}(a)$
adds $a$ times the j-th row to the i-th row and $-a$ times the (n+i)-th row to
the (n+j)-th row when $i\ne j$ and multiplies the i-th row by $a$ and the 
(n+i)-th row by $a^{-1}$ when $i=j$.

These matrices are products of elementary matrices :
\begin{equation}
K_{ii}(a) = E_{ii}(a-1)F_{ii}(1)E_{ii}(a^{-1}-1)F_{ii}(-a)
\end{equation}
and if $i\ne j$:
\begin{equation}
K_{ij}(a) = F_{jj}(-a)E_{ij}(1)F_{jj}(a)E_{ii}(a)E_{ij}(-1)
\end{equation}

An element $(x_1,\cdots,x_k) \in R^k$ is called \emph{unimodular} if
\[ \sum_{j=1}^{k} x_jR = R. \]
$R$ is said to have \emph{Bass stable rank} $k$ if $k$ is the smallest
integer such that for any unimodular
$(x_1,\cdots,x_{k+1}) \in R^{k+1}$ there exist $(y_1,\cdots,y_k) \in R^k$
such that $(x_1+y_1 x_{k+1},\cdots,x_k+y_k x_{k+1})$ is also unimodular.
We write $bsr(R)=k$. If no such $k$ exists we set $bsr(R)=\infty$. If 
$bsr(R)=1$, then for any $x_1,x_2 \in R$ such that $x_1R+x_2R=R$, there is
$y\in R$ such that $x_1+yx_2 \in R^*$.

If $R$ is a Banach algebra, then the $n \times n$ matrices with entries in $R$
is a normed vector space in the following way. If $M=(a_{ij})$ is a matrix with entries from $R$
(equipped with a norm $|| \cdot ||$), then $N=(||a_{ij}||)$ is a matrix of positive real numbers.
We can now apply any matrix norm to N and this gives a norm of M. These norms will all be equivalent. We say that $M \in \operatorname{Sp}_{2n}(R) $
is \emph{null-homotopic} if there is a continuous map $M(t)$, $0\le t \le 1$, into 
$\operatorname{Sp}_{2n}(R) $ such that $M(0)=I$ and $M(1)=M$. A matrix $M \in \operatorname{Sp}_{2n}(C(X)) $ is
said to be null-homotopic if $M$ is homotopic to the identity when regarded as a map from $X$ to
$\operatorname{Sp}_{2n}(\mathbb{C})$.

\section{Examples}\label{Examples}

We mention here the main examples from 
\cite{Doubtsov/Kutzschebauch:2019}. The interested reader should consult
that paper for further examples.

\begin{Exa}\label{e:starshaped}
If $\Omega \subset \mathbb{C}^n$ is a bounded star-shaped domain and $A(\Omega)$ is the set of 
holomorphic functions in $\Omega$ which are continuous up to the boundary, then every element
$M \in \operatorname{Sp}_{2n}(A(\Omega))$ is null-homotopic under the homotopy $M(t)(z)=M(tz)$ (assuming 
$\Omega$ is star-shaped with respect to the origin). Hence $\operatorname{Sp}_{2n}(A(\Omega))=\operatorname{Ep}_{2n}(A(\Omega))$
by Theorem \ref{t:Theorem 1}. 

For the disc algebra $A(\mathbb{D})$ this result also follows from Theorem \ref{t:Theorem 2}
since the Bass stable rank of $A(\mathbb{D})$ equals one. (See Jones, Marshall and Wolff 
(\cite{Jones:86}) and Corach and Suarez (\cite{Corach:1985}).) It is known that the Bass
stable rank of the disc and ball algebras in higher dimensions is strictly greater than one,
so these cases do not follow from Theorem \ref{t:Theorem 2}.
\end{Exa}

\begin{Exa}\label{e:O(X)}
If $X$ is an open Riemann surface, then $\mathcal{O}(X)$ has Bass stable rank one. This follows
from the sharpened version of Wedderburn's lemma which can be found in R.Remmert's textbook
(page 137 of \cite{Remmert:1998}). Hence $\operatorname{Sp}_{2n}(\mathcal{O}(X))=\operatorname{Ep}_{2n}(\mathcal{O}(X))$ by
Theorem \ref{t:Theorem 2} and every $M \in \operatorname{Sp}_{2n}(\mathcal{O}(X))$ is null-homotopic. 
This provides an easy proof of the symplectic Vaserstein problem in dimension one.
\end{Exa}

\begin{Exa}\label{e:Hinfinity}
Treil proved that $\mbox{H}^{\infty}(\mathbb{D})$ has Bass stable rank one 
(\cite{Treil:1992}). Hence
$\operatorname{Sp}_{2n}(\mbox{H}^{\infty}(\mathbb{D}))=\operatorname{Ep}_{2n}(\mbox{H}^{\infty}(\mathbb{D}))$ by
Theorem \ref{t:Theorem 2} and every $M \in \operatorname{Sp}_{2n}(\mbox{H}^{\infty}(\mathbb{D}))$ is null-homotopic. 
\end{Exa}

\section{Proof of Theorem \ref{t:Theorem 1}}\label{Proof of theorem 1.1}

In this section $R$ is a commutative Banach algebra with unity.
$\operatorname{Sp}_{2n}(R)$ is a metric space with metric induced by a norm of $M_{2n}(R)$.
The main part of the proof consists in showing that the Gauss-Jordan process can be carried out by 
multiplying by elementary symplectic matrices.  
If we start with a matrix sufficiently close to the identity, there is no 
need to change the order of the rows and the diagonal elements will stay close
to  $1$ during the whole process. It is clear that this process is well defined
and continuous in a neighbourhood of $I \in \operatorname{Sp}_{2n}(R)$ and even holomorphic in case
$R=\mathbb{C}$.

Hence we start with a matrix \[ M = \begin{pmatrix} A & B \\ C & D \end{pmatrix} 
\] sufficiently close to the identity.
We denote $A=(a_{ij})$ and similarly for $B$,$C$ and $D$. 
We shall now multiply successively from the left by 
elementary matrices, but use the same notation for the result,
i.e the entries of the matrices $A,B,C$  and $D$ will change in every step.
The goal is to end up with the identity matrix.

Multiplying by $K_{11}(a_{11}^{-1})$ gives $a_{11}= 1$. We then proceed by 
multiplying by $K_{i1}(-a_{i1})$ for $i=2,\cdots,n$ to achieve $a_{i1}=0$ for $i > 1$. 
Next step is to multiply by $F_{i1}(-c_{i1})$ for $i=1,\cdots,n$ to obtain $c_{i1}=0$ 
for all $i$. We are now done with the first column. It also follows by (\ref{e:firstsymp}) 
that the first row of $C$ is zero.The steps that follow will not affect this column or
row.

We now multiply by $K_{22}(a_{22}^{-1})$ to get $a_{22}=1$. Then multiply by 
$K_{i2}(-a_{12})$ for $i=1,3,\cdots,n$ to get $a_{i2}=0$ for those $i$. Finally 
multiply by $F_{i2}(-c_{i2})$ for $i\ge 2$ to get $c_{i2}=0$ for $i\ge 2$.We 
already know that $c_{12}=0$ so the second column of $C$ is zero and we are
done with the second column. Again by (\ref{e:firstsymp})
it follows that the second row of $C$ is also zero and the first two columns of
$M$ and rows of $C$ are not affected by the remaining steps.

Continuing in this way on the first $n$ columns gives $A=I$ and $C=0$. By (\ref{e:thirdsymp}) 
and (\ref{e:secondsymp}), $D=I$ and $B$ is symmetric. Multiplying by 
$E_{ij}(-b_{ij})$ for $1 \le j \le i \le n$ annihilates $B$ and we get $M=I$ ,
the $2n \times 2n$ identity matrix. We have now proved

\begin{Lem}\label{l:Gauss-Jordan}(Gauss-Jordan process for symplectic matrices)
Let $R$ be a commutative Banach algebra with unity. 
There is a neighbourhood $V$ of the identity in $\operatorname{Sp}_{2n}(R)$ 
and elementary matrices $E_1,\cdots,E_N$ $(N=(N(n))$, depending continuously on 
$M\in V$, such that  $E_i(I)=I$ and $M=E_1\cdots E_N$ for all $M \in V$.
\end{Lem}

\begin{proof}[Proof of Theorem \ref{t:Theorem 1}]
Let $M$ be a null-homotopic matrix in $\operatorname{Sp}_{2n}(R)$ and denote the homotopy by $M_t$.
By uniform continuity of $M_t$ (and a lower bound on $||M_t||$) it follows that there
is a $\delta > 0$ such that $M_tM_{t'}^{-1} \in V$ whenever $|t-t'| < \delta$. 
Hence for $k > \frac{1}{\delta}$ we have
$$
M=M_1= (M_1 M_{1-\frac{1}{k}}^{-1})(M_{1-\frac{1}{k}} M_{1-\frac{2}{k}}^{-1}) \cdots
M_{\frac{1}{k}}
$$
Hence $M$ is a product of $k$ matrices in $V$ and each of these is a product of $N$
elementary matrices by the previous lemma. This completes the proof.
\end{proof}

\section{Proof of Theorem \ref{t:Theorem 2}}\label{Proof of theorem 1.2}

As for the Gauss-Jordan process we start with a matrix \[ M = \begin{pmatrix} A & B \\ C & D \end{pmatrix} \]
and multiply from the left by elementary matrices without changing the notation. The Bass stable
rank condition will allow us to produce invertible pivots so we can proceed with Gauss-Jordan as
above. 

Expanding the determinant along the first column gives the existence of $x_i$ and $y_i$, $1 \le i \le n$
such that 
$$
x_1 a_{11} + \sum_{i=2}^n x_i a_{i1} + \sum_{i=1}^n y_i c_{i1} = 1
$$  
By the Bass stable rank condition there is $\alpha \in R$ such that
$$
a_{11} + \sum_{i=2}^n \alpha x_i a_{i1} + \sum_{i=1}^n \alpha y_i c_{i1} \in R^*
$$  
We now multiply from the left by $K_{1i}(\alpha x_i)$ for $2 \le i \le n $. The first column
now becomes 
$$
(a_{11} + \sum_{i=2}^n \alpha x_i a_{i1}, a_{21},\cdots,a_{n1},c_{11},c_{21}-\alpha x_2 c_{11},
\cdots, c_{n1}-\alpha x_n c_{11})^T
$$
We then multiply by $E_{1i}(\alpha y_i)$ for $2 \le i \le n $. The first element now becomes
$$
a_{11} + \sum_{i=2}^n \alpha x_i a_{i1} + \sum_{i=2}^n \alpha y_i c_{i1} -
\sum_{i=2}^n \alpha^2 x_i y_i c_{11}
$$
and the value of $c_{11}$ does not change. We now multiply by $E_{11}(\alpha y_1 +\sum_{i=2}^n \alpha^2 x_i y_i)$
and the first element becomes
$$
a_{11} + \sum_{i=2}^n \alpha x_i a_{i1} + \sum_{i=1}^n \alpha y_i c_{i1}
$$
which is invertible and we may proceed as in Gauss-Jordan to make the first column equal
to $e_1$. We can now proceed to the next column, sticking to the same notations ($x_i,y_i,\alpha $).
After multiplication by $K_{2i}(\alpha x_i)$ for $3 \le i \le n $ the first column is
$$
(a_{12},a_{22} + \sum_{i=3}^n \alpha x_i a_{i2}, a_{32},\cdots,a_{n2},c_{12},c_{22},c_{32}-\alpha x_3 c_{22}
\cdots, c_{n2}-\alpha x_n c_{22})^T
$$
Multiplying by $E_{2i}(\alpha y_i)$ for $i=1,3,\cdots,n $ produces
$$
a_{22} + \sum_{i=3}^n \alpha x_i a_{i2} + \sum_{i\ne 2} \alpha y_i c_{i2} -
\sum_{i=3}^n \alpha^2 x_i y_i c_{22}
$$
in position $22$ without changing $c_{22}$. Finally we multiply by
$E_{22}(\alpha y_2 +\sum_{i=3}^n \alpha^2 x_i y_i)$ to produce an invertible
element in position $22$ and we may proceed with Gauss-Jordan. It is clear that we 
can continue this process and complete the proof as in the Gauss-Jordan process.

\section{Proof of Theorem \ref{t:Theorem 3}}\label{Proof of Theorem 1.3}

The proof consists of three ingredients; the Gauss-Jordan elimination result for 
$R=\mathbb{C}$, the Gram-Schmidt process for complex symplectic matrices and a 
result on uniform homotopies by Calder and Siegel (\cite{Calder/Siegel:1978},\cite{Calder/Siegel:1980}).

Let us first see how to carry out the Gram-Schmidt process for a matrix $M \in \operatorname{Sp}_{2n}(\mathbb{C})$.
Let \[v_1,\cdots,v_n,w_1,\cdots,w_n\] denote the rows of $M$. We shall now proceed to multiply $M$ by the elementary matrices introduced above, but will still refer to the result by the same notation, i.e. $M$ and $v_1,\cdots,w_n$ will change in every step.

The first step is to make all the $v$'s orthogonal. Multiplication by $K_{i1}(\frac{-<v_i,v_1>}{||v_1||^2})$ for $i=2,\cdots,n$ removes the components of $v_2,\cdots,v_n$ along $v_1$, i.e. we get $v_i \perp v_1$  for $i \ge 2$. This also changes the $w$'s. We can now continue to multiply by $K_{i2}(\frac{-<v_i,v_2>}{||v_2||^2})$ for $i \ge 3$, etc. The end result makes all the $v$'s orthogonal.

In the next step we make the $v$'s orthonormal by multiplying by $K_{ii}(\frac{1}{||v_i||})$ for $i=1,\cdots,n$. Notice that the $w$'s change in all the above steps.

In the final step we make $w_j$ orthogonal to $v_i$ for $i\ge j$. Starting with $w_1$, we multiply by $F_{1j}(-<w_1,v_j>)$ for $j=1,\cdots,n$ to make $w_1$ orthogonal to all the $v$'s. This changes $w_2,\cdots,w_n$. We then continue to multiply by $F_{2j}(-<w_2,v_j>)$ for $j=2,\cdots,n$ to make $w_2$ orthogonal to $v_2,\cdots,v_n$. This changes $w_3,\cdots,w_n$, but not $w_1$. Continuing like this produces the desired result.

We shall  see that $M$ is now in $\operatorname{SU}(2n) $. The matrix $MM^{*}$ is symplectic since
$\operatorname{Sp}_{2n}(\mathbb{C})$ is closed under transposition and complex conjugation. 
It is also Hermitian and satisfies $A=I$ by construction. By the final step $C$ has zeroes 
on and above the diagonal. By (\ref{e:firstsymp}), $C=C^t$ hence $C=0$. Since $MM^*$ is 
Hermitian it follow that $B=C^*=0$. Finally, (\ref{e:thirdsymp}) gives us that $D=I$.

It is clear from the construction that all the matrices we used to multiply our original matrix by depend continuously on the initial matrix. Denoting the compact symplectic group $\operatorname{Sp}_{2n}(\mathbb{C}) \cap \operatorname{U}(2n)$ by $\operatorname{Sp}(n)$ we have now proved:

\begin{Lem}\label{l:Gram-Schmidt}(Gram-Schmidt process for symplectic matrices) For every integer $n$ there is an integer $L(=L(n))$ and elementary symplectic matrices $F_1,\cdots,F_L$, depending continuously on $M \in \operatorname{Sp}_{2n}
(\mathbb{C})$ such that $F_1\cdots F_LM \in \operatorname{Sp}(n)$ for all $M$.
\end{Lem}

The final ingredient in the proof of Theorem \ref{t:Theorem 3} is a version of a result of Calder and Siegel (\cite{Calder/Siegel:1978}, \cite{Calder/Siegel:1980}). Here $||\cdot||$ denotes any matrix norm. Since the compact symplectic group $\operatorname{Sp}(n)$ is simply connected 
(Proposition 13.12, \cite{Hall:2015}), we get the following result.


\begin{Thm}\label{t:Calder/Siegel}(Calder/Siegel) Let $X$ be a finite dimensional normal 
space and assume $M\colon X\to \operatorname{Sp}(n)$ is null-homotopic. Then there is a 
uniform homotopy $M_t\colon X\to \operatorname{Sp}(n)$ with $M_1=M$ and $M_0=I$, i.e. 
for any $\epsilon > 0$ there is a $\delta > 0$ such that $||M_t(x)-M_{t^{'}}(x)||< \epsilon $ 
for all $x \in X$ and $|t-t^{'}|<\delta $.
\end{Thm}

By writing \[M=(M_1M_{\frac{k-1}{k}}^{-1})(M_{\frac{k-1}{k}}M_{\frac{k-2}{k}}^{-1})\cdots M_{\frac{1}{k}}\] 
for some large $k$ it follows that for any $\epsilon >0$ there are finitely many continuous 
matrices $N_1,\cdots,N_k$ in $\operatorname{Sp}(n)$ such that $M=N_1\cdots N_k$ and $||I - N_j(x)|| 
< \epsilon $ for all $x\in X$ and $j$. We are now ready to prove Theorem \ref{t:Theorem 3}.

\begin{proof}[Proof of Theorem \ref{t:Theorem 3}]

Let $P_t$ denote the null-homotopy, i.e. $P_t \colon X \to \operatorname{Sp}_{2n}(\mathbb{C})$ 
with $P_1=M$ and $P_0=I$. By Lemma \ref{l:Gram-Schmidt} there are elementary symplectic 
matrices $F_1,\cdots,F_L$ such that $V_t = F_1(P_t)F_2(P_t)\cdots F_L(P_t)P_t$ is a null-homotopy 
with values in $\operatorname{Sp}(n)$ such that $V_1 = F_1(M)\cdots F_L(M)M$.

By Theorem \ref{t:Calder/Siegel} there is a uniform null-homotopy $M_t \colon X \to \operatorname{Sp}(n)$ with \[M_1=F_1(M)\cdots F_L(M)M\] and by the above comment there are finitely many continuous matrices $N_1,\cdots,N_k$ in $ \operatorname{Sp}(n)$ such that $M_1(x)=N_1(x)\cdots N_k(x)$ for all $x\in X$ and we may choose $k$ such that all values $N_j(x)$ lie in the neighbourhood $V$ of Lemma \ref{l:Gauss-Jordan}.

It now follows that we can write \[ F_1(M(x))\cdots F_L(M(x))M(x)= \prod_{j=1}^k \prod_{i=1}^N E_i (N_j(x))\] hence this gives us \[ M(x)= F_L^{-1}(M(x))\cdots F_1^{-1}(M(x)) \prod_{j=1}^k \prod_{i=1}^N E_i (N_j(x))\]

All the matrices on the right-hand side are elementary symplectic matrices depending continuously on $x\in X$. This completes the proof of the theorem.
\end{proof}

\end{document}